\documentclass[12pt,reqno]{amsart}
\usepackage{amsfonts,amsmath,amssymb,amsthm}
\usepackage{enumitem}
\usepackage{amsaddr}
\usepackage{mathrsfs}
\usepackage{mathabx}
\usepackage{bbm}
\usepackage{cite}
\newtheorem{Theorem}{Theorem}[section]
\newtheorem{Lemma}{Lemma}[section]

\newtheorem{Corollary}{Corollary}[section]

\newtheorem{Definition}{Definition}[section]
\numberwithin{equation}{section}

\usepackage[a4paper,margin=3cm]{geometry}
\begin{document}
\sloppy
\title[Riesz transforms and the pressure in the Navier-Stokes equations]
{On the use of the Riesz transforms to determine the pressure term in the 
incompressible Navier-Stokes equations on the whole space}

\author[B. \'Alvarez-Samaniego]{Borys \'Alvarez-Samaniego}
\address{\vspace{-8mm}N\'ucleo de Investigadores Cient\'{\i}ficos\\
	    Facultad de Ciencias\\
	    Universidad Central del Ecuador (UCE)\\
    	Quito, Ecuador}
\email{balvarez@uce.edu.ec, borys\_yamil@yahoo.com}

\author[W. P. \'Alvarez-Samaniego]{Wilson P. \'Alvarez-Samaniego}
\address{\vspace{-8mm}N\'ucleo de Investigadores Cient\'{\i}ficos\\
        Facultad de Ciencias\\
	    Universidad Central del Ecuador (UCE)\\
    	Quito, Ecuador}
\email{wpalvarez@uce.edu.ec, alvarezwilson@hotmail.com}

\author[P. G. Fern\'andez-Dalgo]{Pedro Gabriel Fern\'andez-Dalgo}
\address{\vspace{-8mm}N\'ucleo de Investigadores Cient\'{\i}ficos\\
        Facultad de Ciencias\\
    	Universidad Central del Ecuador (UCE)\\
    	Quito, Ecuador \\
    	\vspace{2mm}
    	LaMME, Universit\'e d'Evry Val d'Essonne, CNRS\\ 
    	Universit\'e Paris-Saclay\\
    	91025, Evry, France}
\email{pedro.fernandez@univ-evry.fr}

\date{March 23, 2020} 

\begin{abstract}
We give some conditions under which the pressure term in the incompressible 
Navier-Stokes equations on the entire $d$-dimensional Euclidean space is 
determined by the formula $\displaystyle \nabla p = \nabla \left(\sum_{i,j=1}^d 
\mathcal{R}_i \mathcal{R}_j (u_i u_j - F_{i,j}) \right)$, where $d \in \{2, 3\}$, 
${\textbf{u}} := (u_1, \ldots ,u_d)$ is the fluid velocity, 
$\mathbb{F}:= (F_{i,j})_{1\le i,j\le d}$ is the forcing tensor, and for all 
$k \in \{1, \ldots, d\}$, $\mathcal{R}_k$ is the $k$-th Riesz transform.
\end{abstract}


\maketitle
\section{Introduction} \label{sec:intro}
The incompressible Navier-Stokes equations on the Euclidean space $\mathbb{R}^d$, 
with $d \in \mathbb{Z}^+$, are the key governing equations of viscous fluid flows 
with a divergence-free fluid velocity vector field occupying all the 
$\mathbb{R}^d$ space, which are given by 
\begin{equation*}  \label{Eq:NS}
  \tag{NS}
  \left\{
	\begin{array}{l}
	  \partial_t {\textbf{u}} = \Delta {\textbf{u}} - ({\textbf{u}} \cdot \nabla)
      {\textbf{u}}- \nabla p +\nabla \cdot \mathbb{F},  \\  \\
      \nabla \cdot {\textbf{u}}  =0, 
    \end{array}
  \right.
\end{equation*}
where ${\textbf{u}} := (u_1, \ldots ,u_d)$ is the fluid velocity, $p$ is the fluid 
pressure, $\mathbb{F}:= (F_{i,j})_{1\le i,j\le d}$ is the forcing tensor,   
$\displaystyle \nabla \cdot \mathbb{F}:= \left (\sum_{i=1}^d \partial_i F_{i,1}, 
\ldots, \sum_{i=1}^d \partial_i F_{i,d} \right)$, and 
$\displaystyle \nabla \cdot \textbf{u}:= \sum_{i=1}^d \partial_i u_i$.

We now introduce several notations, preliminary results and definitions 
that will be helpful in the sequel.  For all 
$y := (y_1, \ldots, y_d) \in \mathbb{R}^d$, we write 
$|y|:= \sqrt{|y_1|^2 + \cdots + |y_d|^2}$.  For any $d \in \mathbb{Z}^+$,  
and for every $\gamma > 0$, we denote by $w_{\gamma}:=w_{d, \gamma}$, 
the weight function, provided by
\begin{align*}
   w_{\gamma} \colon \mathbb{R}^d &\longrightarrow \mathbb{R}  \\
   x &\longmapsto w_\gamma (x):= \frac{1}{(1+|x|)^{\gamma}}.
\end{align*}
For all $d \in \mathbb{Z}^+$, for every $\delta > 0$, and for each 
$1\le p <+\infty$, we use the notation $L^p_{w_\delta}(\mathbb{R}^d)$ 
to represent the weighted $L^p$-space given by 
$L^p(\mathbb{R}^d, w_\delta (x) \, dx)$.  Moreover, for all 
$d \in \mathbb{Z}^+$, and for any $j \in \{1, \ldots, d \}$, 
$\mathcal{R}_j:=\mathcal{R}_{d, j}$ denotes the $j$-th \emph{Riesz transform}, 
given by 
\begin{equation*}
  \mathcal{R}_j := \frac{\partial_j}{\sqrt{-\Delta}}.
\end{equation*}  
Let ${\textbf{b}}$ and ${\textbf{u}}$ be two vector fields on 
$\mathbb{R}^d$.  The \textit{tensor product}, 
${\textbf{b}} \otimes {\textbf{u}}$, of ${\textbf{b}}$ and 
${\textbf{u}}$, is defined as the $d \times d$-matrix, given, 
for all $i, j \in \{1, \ldots, d\}$, by 
$({\textbf{b}} \otimes {\textbf{u}})_{i, j} := b_i u_j$.  Thus, 
if $\nabla \cdot {\textbf{b}} = 0$, we have that 
$({\textbf{b}} \cdot \nabla) {\textbf{u}} 
= \nabla \cdot ({\textbf{b}}\otimes {\textbf{u}})$.

Recently, P.G. Fern\'andez-Dalgo and P.G. Lemarié-Rieusset 
(\hspace{1sp}\cite{PF_PG2}) gave a general characterization for the pressure 
term in the incompressible Navier-Stokes equations (\ref{Eq:NS}) on the whole 
Euclidean space $\mathbb{R}^d$, with $d \in \{2,3\}$. In \cite{PF_PG2}, the 
authors consider velocities $\textbf{u} = (u_1,...,u_d)$ belonging to 
$L^2 ((0,T), L^2_{w_\gamma}(\mathbb{R}^{d}))$, with $\gamma \in \{ d , d+1 \}$. 
In addition, if $\textbf{u} \in L^2 ((0,T), L^2_{w_d}(\mathbb{R}^d))$, and 
$\mathbb{F} (t,x) := \left(F_{i,j}(t,x)\right)_{1\leq i,j\leq d}$ 
belongs to $L^1\left((0,T), L^1_{w_{d}}(\mathbb{R}^d)\right)$, 
it was also shown in \cite{PF_PG2} that $p$ is given, up to constants, by
\begin{equation*}
  p=\sum_{i,j=1}^d(\varphi \, \partial_i \partial_j G_d) * (u_iu_j-F_{i,j}) +   
  \sum_{i,j=1}^d((1-\varphi) \, \partial_i \partial_j G_d) * (u_iu_j-F_{i,j}),
\end{equation*} 
where $\varphi \in C_0^{\infty}(\mathbb{R}^d) := 
\{ \psi:\mathbb{R}^d \to \mathbb{C}; \; \psi \in C^{\infty}(\mathbb{R}^d) 
\text{ and } \psi \text{ has compact support} \}$ 
is a real-valued function such that $\varphi = 1$ on 
a neighborhood of $0$, $*$ stands for the \textit{convolution product}, 
and $G_d$ is the \textit{fundamental solution of the Laplace equation} on 
$\mathbb{R}^d$, i.e., $-\Delta G_d=\delta$.  So, for all 
$x \in \mathbb{R}^d \smallsetminus \{0\}$, with $d \in \{2, 3\}$,
\begin{equation*} 
  G_2(x) = \frac 1{2\pi}\ln \left(\frac 1{\vert x\vert}\right)
  \quad \text{ and } \quad G_3(x) = \frac 1{4\pi\vert x\vert}.
\end{equation*}  
Furthermore, in the more favorable case, when 
$\textbf{u} \in L^\infty ((0,T),  L^2_{w_\gamma}(\mathbb{R}^3))$, 
$\nabla {\textbf{u}} \in L^2((0,T), L^2_{w_\gamma}(\mathbb{R}^3))$, and
$\mathbb{F} \in L^2((0,T), L^2 _{w_{\gamma} }(\mathbb{R}^3))$, 
with $0 < \gamma < \frac{5}{2}$, the authors in \cite{PF_PG} showed 
that $p$ is determined, up to constants, by the more simple formula 
\begin{equation*}
  p = \sum_{i,j=1}^3 \mathcal{R} _i \mathcal{R} _j (u_i u_j - F_{i,j}).
\end{equation*}

The following definition concerns the \textit{Muckenhoupt class} of 
weights.
\begin{Definition}
Let $d \in \mathbb{Z}^+$, and $1<p<+\infty$.  A nonnegative real-valued 
measurable function, $w :\mathbb{R}^d \to [0, + \infty)$, belongs to the 
\emph{Muckenhoupt} $\mathcal{A}_p(\mathbb{R}^d)$ \emph{class} if it is 
locally integrable and it satisfies the \emph{reverse H\"older inequality},  
with conjugate exponents $p$ and $q:= \frac{p}{p-1}$, given by 
\begin{equation} \label{Eq:muckenhoupt}
  \sup_{x\in\mathbb{R}^d, R>0}  
  \left(\frac 1{\vert B(x,R)\vert}  \int_{B(x,R)} 
  \! \! \! \! \! w(y)\, dy \right)^{\frac 1 p}
  \cdot
  \left( \frac 1{\vert B(x,R)\vert} \int_{B(x,R)} 
  \! \! \! \! \! w(y)^{-\frac 1{p-1}} dy 
  \right)^{1-\frac 1 p}  \!\!<+\infty.
\end{equation}
\end{Definition}
The next result is a $d$-dimensional generalization of Lemma 1 in 
\cite{PF_PG}, where it was assumed that $d=3$.
\begin{Lemma}[Muckenhoupt weights] \label{Lemma:Muck} 
Let $d \in \mathbb{Z}^+$, $0<\delta<d$, and $1<p<+\infty$.  Then,  
$w_\delta$ belongs to the Muckenhoupt $\mathcal{A}_p(\mathbb{R}^d)$ 
class.
\end{Lemma}
\begin{proof} 
We suppose that $d \in \mathbb{Z}^+$, $0<\delta<d$, and $1<p<+\infty$. 
Let $x\in \mathbb{R}^d$. Let $y \in \mathbb{R}^d$ be such that 
$\vert y-x\vert<R$.  \\
{\bf{(i)}} First, we assume that $R\in(0, 1]$. Then,
$1+\vert y\vert \le 1 + \vert x-y \vert + \vert x \vert 
<  1 + R + \vert x \vert \le 2 + \vert x \vert \le 2(1 + \vert x \vert)$.
Proceeding similarly as in the last expression, we get 
$1+ \vert x \vert < 2(1 + \vert y \vert)$.  Hence, 
\begin{equation} \label{Eq:Rsmall}
  \frac 1 2 (1+\vert x\vert) < 1+\vert y\vert < 2 (1+\vert x\vert). 
\end{equation}
It follows from (\ref{Eq:Rsmall}) that 
\begin{align*}
     & \left(\frac 1{\vert B(x,R)\vert}  
       \int_{B(x,R)}  w_{\delta}(y)\, dy\right)^{\frac 1 p} 
       \cdot \left( \frac 1{\vert B(x,R)\vert} 
       \int_{B(x,R)} w_{\delta}(y)^{-\frac 1{p-1}} \, dy \right)^{1-\frac 1 p}  \\
 \le & \left( \frac{2^{\delta}}{(1+ \vert x \vert)^{\delta}} \right)^{\frac{1}{p}}
       \cdot \left( 2^{\frac{\delta}{p-1}} \cdot 
       (1+\vert x \vert)^{\frac{\delta}{p-1}} \right)^{\frac{p-1}{p}}
       = 4^{\frac{\delta}{p}}.
\end{align*}

\noindent
{\bf{(ii)}} We now pretend that $R > 1$. 
\begin{itemize}
\item  We first consider the case when 
$\vert x \vert > 10R$.  Then, 
$1+\vert y \vert \le 1 + \vert x-y \vert + \vert x \vert 
< 1 + R + \vert x \vert < 1 + \left( \frac{1}{10} + 1 \right) \vert x \vert  
\le \frac{11}{10}(1 + \vert x \vert)$.  Moreover, 
$1+ \vert x \vert \le  1 + \vert x-y \vert + \vert y \vert
< 1 + R + \vert y \vert < 1 + \frac{1}{10} \vert x \vert + \vert y \vert$.
Thus, $\frac{9}{10} \left( 1 + \vert x \vert \right) 
< 1 + \frac{9}{10} \vert x\vert < 1 + \vert y \vert$.  So, 
\begin{equation} \label{Eq:Rbig-xbig}
  \frac{9}{10} \left( 1 + \vert x \vert \right) < 1 + \vert y \vert 
  < \frac{11}{10} \, (1 + \vert x \vert).
\end{equation}
By using (\ref{Eq:Rbig-xbig}), we get  
\begin{align*}
     & \left(\frac 1{\vert B(x,R)\vert}  
       \int_{B(x,R)}  w_{\delta}(y)\, dy\right)^{\frac 1 p} 
       \cdot \left( \frac 1{\vert B(x,R)\vert} 
       \int_{B(x,R)} w_{\delta}(y)^{-\frac 1{p-1}} \, dy \right)^{1-\frac 1 p}  \\
 \le & \left[ \left(\frac{10}{9} \right)^{\delta}
       \cdot \frac{1}{(1+ \vert x \vert)^{\delta}} \right]^{\frac{1}{p}}
       \cdot \left[ \left( \frac{11}{10} \right)^{\frac{\delta}{p-1}} 
       \cdot (1+\vert x \vert)^{\frac{\delta}{p-1}} \right]^{\frac{p-1}{p}}
       = \left( \frac{11}{9} \right)^{\frac{\delta}{p}}.
\end{align*} 
 
\item Finally, we presume that $\vert x\vert \le 10 R$.  Since 
$|B(0,R)| = |B(x,R)|$, and $B(x,R ) \subset B(0,11R)$, we obtain that 
\begin{align*} 
     & \left(\frac 1{\vert B(x,R)\vert}  
       \int_{B(x,R)}  w_{\delta}(y)\, dy\right)^{\frac 1 p} 
       \cdot \left( \frac 1{\vert B(x,R)\vert} 
       \int_{B(x,R)} w_{\delta}(y)^{-\frac 1{p-1}} \, dy \right)^{1-\frac 1 p}  \\
 \le & \left(\frac {1}{\vert B(0,R)\vert}  
       \int_{B(0,11 R)} w_\delta (y)\, dy\right)^{\frac 1 p}
       \cdot \left( \frac {1}{\vert B(0,R)\vert} 
       \int_{B(0,11 R)} w_\delta (y)^{-\frac 1{p-1}} \, dy\right)^{1-\frac 1 p} \\ 
 =   & \left(\frac 1 { R^d} \int_0^{11\, R} \frac{1}{(1+r)^\delta} 
       \cdot r^{d-1} \, dr\right)^{\frac 1 p} 
       \cdot \left(\frac{1}{R^d} \int_0^{11R} (1+r)^{\frac{\delta}{p-1}}
       \cdot r^{d-1} \, dr\right)^{1-\frac{1}{p}} \\
 \le & \; c_{\delta, p} \left(\frac 1 {R^d} \int_0^{11R}  
      \frac{1}{r^\delta} \cdot r^{d-1} \, dr \right)^{\frac{1}{p}}  \\
     & \cdot \left[\left(\frac{1}{R^d} \int_0^{11R} r^{d-1} \,dr \right)^{1-\frac{1}{p}} 
       + \left( \frac 1 {R^d} \int_0^{11R} 
       r^{\frac{\delta}{p-1}+d-1} \, dr\right)^{1-\frac{1}{p}} \right]  
\end{align*}
\begin{align*}          
 =   & \; c_{\delta, p} \cdot \left( \frac{11^{d-\delta}}
       {R^{\delta} \cdot (d-\delta)} \right)^{\frac{1}{p}}    \cdot 
       \left[ \left(  \frac{11^d}{d} \right)^{1-\frac{1}{p}} 
       + \left( \frac{11^{\frac{\delta}{p-1}+d} \cdot R^{\frac{\delta}{p-1}}}
       {\frac{\delta}{p-1} +d} \right)^{1-\frac{1}{p}} \right]  \\
 =   & \; c_{\delta, p} \cdot \frac{11^d}{(d-\delta)^{\frac{1}{p}}} \cdot 
       \left( \frac{(11R)^{-\frac{\delta}{p}}}{d^{1-\frac{1}{p}}} 
       + \frac{1}{\left( \frac{\delta}{p-1} + d \right)^{1-\frac{1}{p}}} \right) \\
 \le & \; c_{\delta, p} \cdot \frac{11^d}{(d-\delta)^{\frac{1}{p}}} \cdot 
       \left( \frac{1}{d^{1-\frac{1}{p}}} 
       + \frac{1}{\left( \frac{\delta}{p-1} + d \right)^{1-\frac{1}{p}}} \right).   
\end{align*}
\end{itemize}
From {\bf{(i)}} and {\bf{(ii)}}, the lemma follows. 
\end{proof}
The following result will be used in the proof of Theorem~\ref{Theo:pr}.
\begin{Corollary} \label{Cor:operonb}
Let $d \in \mathbb{Z}^+$, $0<\delta<d$, and $1<p<+\infty$. Then, the 
following assertions hold.
\begin{itemize} 
\item For all $j \in \{1, \ldots, d \}$, the $j$-th \emph{Riesz transform},  
$\mathcal{R}_j := \frac{\partial_j}{\sqrt{-\Delta}}$, is a bounded linear 
operator on $L^p_{w_\delta}(\mathbb{R}^d)$, i.e., there exists a constant 
$C:=C_{d, \delta, p} > 0$ such that for all $f \in L^p_{w_\delta}(\mathbb{R}^d)$, 
\begin{equation*}
  \|\mathcal{R}_j f\|_{L^p_{w_\delta}(\mathbb{R}^d)} 
  \le \, C \, \|f\|_{L^p_{w_\delta}(\mathbb{R}^d)}.
\end{equation*}
\item The \emph{Hardy-Littlewood maximal function operator} is a 
bounded nonlinear operator on $L^p_{w_\delta}(\mathbb{R}^d)$, i.e., 
there is a constant $K:=K_{d, \delta, p} > 0$ such that for all 
$g \in L^p_{w_\delta}(\mathbb{R}^d)$, 
\begin{equation*}
  \|\mathcal{M} g\|_{L^p_{w_\delta}(\mathbb{R}^d)} 
  \le \, K \, \|g\|_{L^p_{w_\delta}(\mathbb{R}^d)}.
\end{equation*}
\end{itemize}
\end{Corollary}
\begin{proof}
These are properties of the Muckenhoupt $\mathcal{A}_p(\mathbb{R}^d)$ class  
(we refer to \cite{Gr09}).
\end{proof}
The main result of this paper and its proof are presented in the next 
section.  It is an improvement of Corollary 2 in \cite{PF_PG}, where 
it was given a characterization of the pressure term, $\nabla p$, in the 
incompressible 3D Navier-Stokes equations (\ref{Eq:NS}) on the whole 
Euclidean space $\mathbb{R}^3$. More specifically, Theorem~\ref{Theo:pr} 
below considers the dimension of the Euclidean space $d\in \{2,3\}$, 
while Corollary 2 in \cite{PF_PG} refers only to the case of dimension 3.  
Moreover, in the hypotheses of Theorem~\ref{Theo:pr}, for $d=3$, 
we suppose that $\gamma \in (0, 3)$, where $\gamma$ is the parameter of 
the weight function, $w_\gamma$, whereas Corollary 2 in \cite{PF_PG} 
concerns the case $\gamma \in (0,5/2)$.  It deserves to remark that 
the hypotheses of Theorem~\ref{Theo:pr} imply the assumptions of 
Theorem 1 in \cite{PF_PG2}; however, Theorem~\ref{Theo:pr} allows us 
to manipulate the pressure term in the Navier-Stokes equations 
more easily, by using the Riesz transforms, than employing the techniques 
of Theorem 1 in \cite{PF_PG2}. Furthermore, we notice that in 
a recent work, Z. Bradshaw and T.P. Tsai (\hspace{1sp}\cite{BTpr}) study 
local expansions for the pressure term in the Navier-Stokes equations. To 
end this section, we observe that some related results can also be found 
in \cite{BTK}, \cite{CW18, PF_OJ}, and 
\cite{JS14, KS07, LR99, LR02, LR16, Le34, VF, Wo17}.

\section{Determination of the pressure term in the incompressible 
Navier-Stokes equations on the full Euclidean space}
The primary goal of this section is Theorem~\ref{Theo:pr} below. It can be 
used to obtain a priori controls for solutions of the 2D incompressible 
Navier-Stokes equations in weighted $L^2$-spaces, and to study existence, 
uniqueness, and regularity of these solutions.  We note that A. Basson 
(\hspace{1sp}\cite{Ba06}) proved uniqueness of the solutions of the 2D 
incompressible Navier-Stokes equations, for uniformly locally square 
integrable initial data.  It was also proposed in \cite{Ba06} the open 
problem concerning the uniqueness of solutions of the 2D incompressible 
Navier-Stokes equations, for initial data ${\textbf{u}}_0$ belonging to 
the space $B_2(\mathbb{R}^2)$, where 
\begin{equation*}
    \| {\textbf{u}}_0 \|_{B_2(\mathbb{R}^2)} ^2 :=  
    \sup_{R \ge 1}  \frac{1}{R^2} \int_{|y|\le R} 
    |{\textbf{u}}_0 (y)|^{2} \, dy  <  +\infty.
\end{equation*}
It deserves to remark that our estimates for the pressure term, given in 
(\ref{Eq:2}) and (\ref{Eq:4}) below, may be used to study the problem regarding 
the uniqueness of solutions of the 2D incompressible Navier-Stokes equations, 
when the initial data belongs to the weighted $L^2$-space, 
$L^2_{w_\gamma}(\mathbb{R}^2)$, with $0<\gamma<2$.  With respect to the 
aforementioned spaces, we observe that 
(see Section 7 in \cite{PF_PG2}) 
\begin{equation*}
  L^2_{w_{\gamma}}(\mathbb{R}^2) \subset B_{2}(\mathbb{R}^2) 
  \subset L^p_{w_{\delta}}(\mathbb{R}^2), 
\end{equation*}
for $0 < \gamma \leq 2 < \delta$. Now, we go to the major theorem of this 
manuscript.

\begin{Theorem}   \label{Theo:pr}
Let $d \in \{ 2 , 3\}$, $0<\gamma<d$, and $0<T<+\infty$. We suppose 
that $\mathbb{F}:=\left(F_{i,j}\right)_{1\le i,j\le d} 
\in L^2((0,T), L^2_{w_{\gamma}}(\mathbb{R}^d))$. Let $\textup{\textbf{u}}$ be a 
solution of the problem 
\begin{equation}  \label{Eq:NSpr} 
 \left\{ 
   \begin{array}{l} 
   \partial_t \textup{\textbf{u}} = \Delta \textup{\textbf{u}} 
   - (\textup{\textbf{u}} \cdot \nabla) \textup{\textbf{u}} 
   -  \nabla q + \nabla \cdot \mathbb{F}, \\  \\
   \nabla \cdot \textup{\textbf{u}} = 0,
   \end{array}
   \right.
\end{equation}
such that $\textup{\textbf{u}} \in L^\infty((0,T), L^2_{w_\gamma}(\mathbb{R}^d))$,  
and $\nabla \textup{\textbf{u}} \in L^2((0,T), L^2_{w_\gamma}(\mathbb{R}^d))$. 
We also assume that the pressure, $q$, belongs to the space of distributions 
$\mathcal{D}'((0,T) \times \mathbb{R}^d)$. Then, the gradient of the 
pressure is given by the formula
\begin{equation*}
  \nabla  q =  \nabla \left(\sum_{i,j=1}^d \mathcal{R}_i 
  \mathcal{R}_j(u_i u_j - F_{i,j}) \right).
\end{equation*} 
If we, moreover, suppose that 
$r \in \left(1, \min \{ \frac{d}{d-1}, \frac{d}{\gamma} \}\right)$, 
and $a \in \mathbb{R}$ satisfies $\frac{2}{a} + \frac{d}{r} = d$, we get
\begin{equation*}
  \sum_{i,j=1}^d \mathcal{R}_i \mathcal{R}_j (u_i u_j) 
  \in L^{a}((0,T),L^{r}_{w_{ r \gamma}}(\mathbb{R}^d)), 
\end{equation*}
and  
\begin{equation*}
  \sum_{i,j=1}^d \mathcal{R}_i \mathcal{R}_j F_{i,j} 
  \in L^{2}((0,T),L^{2}_{w_\gamma}(\mathbb{R}^d)).
\end{equation*}
\end{Theorem}
\begin{proof}
We let 
\begin{equation*}
 p := \sum_{i,j=1}^d \mathcal{R}_i \mathcal{R}_j (u_i u_j - F_{i,j}).
\end{equation*}
We will show that $\nabla q - \nabla p = 0$, where $q$ is the pressure term 
in the incompressible Navier-Stokes equations on $\mathbb{R}^d$ given in 
(\ref{Eq:NSpr}). We begin by demonstrating that $p$ is well-defined. Let  
$ 1 < r < \min \{ \frac{d}{d-1}, \frac{d}{\gamma} \}$. We take 
$a \in \mathbb{R}$ such that $\frac{2}{a} + \frac{d}{r} = d$, 
and $b \in \mathbb{R}$ such that $\frac{1}{r} = \frac{1}{2} + \frac{1}{b}$. 
Thus, $\frac{2}{a} = d - \frac{d}{r} = \frac{d}{2} - \frac{d}{b}$, and   
$0 < d - \frac{d}{r} < 1$. \\
As $\sqrt{w_\gamma} \, {\textbf{u}} \in L^\infty((0,T), L^2(\mathbb{R}^d))$, 
and $\sqrt{w_\gamma} \, \nabla {\textbf{u}} \in L^2((0,T), L^2(\mathbb{R}^d))$, 
by using the Gagliardo-Nirenberg interpolation inequality, we obtain that 
$\sqrt{w_\gamma} \, {\textbf{u}} \in L^a((0,T), L^b(\mathbb{R}^d))$.  In fact, 
for almost all $t \in (0, T)$, we get 
\begin{align*}
    \| \sqrt{w_\gamma} \, {\textbf{u}}(t) \|_{L^b(\mathbb{R}^d)} 
    & \le C_{d, b} \; \| \nabla (\sqrt{w_\gamma} \, {\textbf{u}}(t)) \|_{L^2  
      (\mathbb{R}^d)}^{\frac{d}{2}-\frac{d}{b}} \; \cdot \;
      \| \sqrt{w_\gamma} \, {\textbf{u}}(t) 
      \|_{L^2 (\mathbb{R}^d)}^{1-(\frac{d}{2}-\frac{d}{b})}  \\
    & \le \gamma \; C_{d, b} 
      \left(\| \sqrt{w_\gamma} \, {\textbf{u}}(t) \|_{L^2(\mathbb{R}^d)} 
      +  \| \sqrt{w_\gamma} \, \nabla {\textbf{u}}(t) \|_{L^2(\mathbb{R}^d)} 
      \right)^{\frac{d}{2}-\frac{d}{b}}  \\
    &  \;\;\;\; \cdot \, \| \sqrt{w_\gamma} \, {\textbf{u}}(t) 
      \|_{L^2 (\mathbb{R}^d)}^{1-(\frac{d}{2}-\frac{d}{b})}.
\end{align*}
As $\frac{2}{a} = \frac{d}{2} - \frac{d}{b}$, and integrating with respect 
to time, we see that
\begin{align} \label{Eq:0a}
  \int_0^T \| \sqrt{w_\gamma} \, {\textbf{u}}(s) \|_{L^b(\mathbb{R}^d)}^a ds
  & \le \gamma \; C_{d, b} \; 
    \| \sqrt{w_\gamma} \, {\textbf{u}} \|_{L^\infty((0,T), L^2(\mathbb{R}^d)) }^{a-2} 
    \nonumber \\
  & \;\;\;\; \cdot \int_0^T ( \| \sqrt{w_\gamma} \, {\textbf{u}}(s) \|_{L^2(\mathbb{R}^d)} 
    + \| \sqrt{w_\gamma} \, \nabla {\textbf{u}}(s) \|_{L^2(\mathbb{R}^d)})^2 ds \\ 
  & < +\infty. \nonumber 
\end{align}
Since $\sqrt{w_\gamma} \, {\textbf{u}} \in L^a((0,T), L^b(\mathbb{R}^d))$, and  
using H\"older's inequality with indices $\frac{2}{r}$ and $\frac{b}{r}$, 
we obtain that for all $i,j \in \{1, \ldots ,d\}$, 
$w_\gamma \, u_i u_j \in L^a((0,T), L^r(\mathbb{R}^d))$, and
\begin{align} \label{Eq:0b}
  \| u_i u_j \|_{L^{a}\left((0,T), L^{r}_{w_{r \gamma}}(\mathbb{R}^d)\right)}
  & = \| \, \| \sqrt{w_\gamma} \, u_i (\cdot) \, \sqrt{w_\gamma} \, 
    u_j (\cdot) \|_{L^{r}(\mathbb{R}^d))} \, \|_{L^{a}(0,T) }  
    \nonumber \\
  & \le \| \, \| \sqrt{w_\gamma} \, u_i (\cdot) \|_{L^{2}(\mathbb{R}^d))} \, \cdot \, 
    \| \sqrt{w_\gamma} \, u_j (\cdot) \|_{L^{b}(\mathbb{R}^d))} \, \|_{L^{a}(0,T) } 
    \nonumber \\
  & \le \| \sqrt{w_\gamma} \, u_i \|_{L^{\infty} \left((0,T), L^{2}(\mathbb{R}^d) \right)} 
    \, \cdot \, 
    \| \sqrt{w_\gamma} \, u_j \|_{L^{a} \left((0,T), L^{b}(\mathbb{R}^d) \right)}.
\end{align}
\noindent
By the continuity of the Riesz transforms on $L^r_{w_{r \gamma}}(\mathbb{R}^d)$ 
(see Corollary~\ref{Cor:operonb} above), and employing the fact that 
$0 < r \, \gamma < d$, we get 
\begin{equation}  \label{Eq:1}
  \sum_{i,j=1}^d \mathcal{R}_i \mathcal{R}_j (u_i u_j)
  \in  L^{a}((0,T), L^{r}_{w_{r \gamma}}(\mathbb{R}^d)),
\end{equation}
and, more precisely, the following estimate holds
\begin{align}  \label{Eq:2}
  & \Big{\|} \, \sum_{i,j=1}^d \mathcal{R}_i \mathcal{R}_j (u_i u_j)  
    \, \Big{\|}_{L^{a}((0,T), L^{r}_{w_{r \gamma}}(\mathbb{R}^d))}
    \nonumber \\
  & \le C_{d, \gamma, r} 
    \sum_{i,j=1}^d \| u_i u_j \|_{L^{a}((0,T), L^{r}_{w_{r \gamma}}(\mathbb{R}^d))} 
    \nonumber \\
  & \le C_{d, \gamma, r} \sum_{i,j=1}^d 
    \|\sqrt{w_\gamma} \, u_i \|_{L^{\infty} \left((0,T), L^{2}(\mathbb{R}^d) \right)} 
    \, \cdot \, 
    \| \sqrt{w_\gamma} \, u_j \|_{L^{a} \left((0,T), L^{b}(\mathbb{R}^d) \right)}
    \nonumber \\
  & \le \gamma^{\frac{1}{a}} \; {\tilde{C}}_{d, \gamma, r}  \; 
    \| {\textbf{u}} \|_{L^{\infty}((0,T), L^{2}_{w_\gamma}(\mathbb{R}^d))}^{1+\frac{a-2}{a}} 
    \nonumber \\
  & \;\;\;\; \cdot     
    \left( \int_0^T ( \| {\textbf{u}}(s) \|_{L^2_{w_\gamma}(\mathbb{R}^d)} 
    + \| \nabla {\textbf{u}}(s) \|_{L^2_{w_\gamma}(\mathbb{R}^d)})^2 ds 
    \right)^{\frac{1}{a}}, 
\end{align}
where we have used (\ref{Eq:0b}) and (\ref{Eq:0a}). Using again the continuity 
of the Riesz transforms, but this time on the space $L^2_{w_{\gamma}}(\mathbb{R}^d)$, 
we obtain that  
\begin{equation}  \label{Eq:3}
  \sum_{i,j=1}^d \mathcal{R}_i \mathcal{R}_j F_{i, j} 
  \in L^{2}((0,T), L^{2}_{w_\gamma}(\mathbb{R}^d)),
\end{equation}
and, moreover, 
\begin{equation} \label{Eq:4}
  \Big{\|}\sum_{i,j=1}^d \mathcal{R}_i \mathcal{R}_j F_{i,j} \, \, 
  \Big{\|}_{L^{2}((0,T), L^{2}_{w_\gamma}(\mathbb{R}^d))} 
  \le C_{d, \gamma} \sum_{i,j=1}^d 
      \| F_{i,j} \|_{L^{2}((0,T), L^{2}_{w_\gamma}(\mathbb{R}^d))}. 
\end{equation}
It follows from (\ref{Eq:1})-(\ref{Eq:4}) that $p$ is well-defined, and 
\begin{equation*}
  p:= \sum_{i,j=1}^d \mathcal{R}_i \mathcal{R}_j (u_i u_j) 
  - \sum_{i,j=1}^d \mathcal{R}_i \mathcal{R}_j F_{i,j} 
  \in L^a((0,T), L^{r}_{w_{s\gamma}}(\mathbb{R}^d)) 
  + L^2 ((0,T), L^2_{w_\gamma}(\mathbb{R}^d )).
\end{equation*}
We now consider the expression $\nabla q - \nabla p$. By taking the 
divergence of both members of the first equation of (\ref{Eq:NSpr}), 
we find that $\Delta (q-p)=0$. In fact, as $\nabla \cdot {\textbf{u}} = 0$,
we have that 
$\nabla \cdot \partial_t {\textbf{u}} = \nabla \cdot \Delta {\textbf{u}} = 0$, 
and
\begin{align} \label{Eq:Lapl}
    \Delta q = \nabla \cdot \nabla q \nonumber 
    & = -\nabla \cdot (({\textbf{u}} \cdot \nabla) {\textbf{u}}
        - \nabla \cdot \mathbb{F}) \nonumber \\
    & = -\nabla \cdot (\nabla \cdot ({\textbf{u}} \otimes {\textbf{u}}) 
        - \nabla \cdot \mathbb{F}) \nonumber \\
    & = -\sum_{i,j=1}^d \partial_i \partial_j 
        (u_i u_j - F_{i,j}) \nonumber \\
    & = \Delta \sum_{i,j=1}^d \frac{\partial_i}{\sqrt{-\Delta}}  
        \frac{\partial_j}{\sqrt{-\Delta}} (u_i u_j - F_{i,j})  \nonumber \\
    & = \Delta p.
\end{align}
Let $\varepsilon_0 \in \left(0, \frac{T}{2}\right)$.  Let 
$\alpha \in C_0^{\infty}(\mathbb{R})$ be a nonnegative real-valued function 
such that for all $t \in \{s \in \mathbb{R};|s| \ge \varepsilon_0 \}$, 
$\alpha (t)= 0$, and  $\int_{\mathbb{R}} \alpha(s) \, ds= 1$.  Moreover, let 
$\beta \in C_0^{\infty}(\mathbb{R}^d)$ be a nonnegative real-valued 
function such that $\int_{\mathbb{R}^d}\beta(x) \, dx = 1$. We use the 
symbol $\alpha\otimes\beta$ to indicate the real-valued function given by
\begin{align*}
   \alpha\otimes\beta \colon \mathbb{R} \times \mathbb{R}^d 
   &\longrightarrow \mathbb{R}  \\
   (s, x) &\longmapsto (\alpha\otimes\beta) (s, x):= \alpha(s) \, \beta(x).
\end{align*}
Also, for all $(s, x) \in \mathbb{R} \times \mathbb{R}^d$, we write 
$(\alpha \otimes \Delta \beta)(s, x) := \alpha(s) \, \Delta \beta(x) 
\in \mathbb{R}$, and $(\alpha \otimes \nabla \beta)(s, x) := (\alpha(s) \, 
\partial_1 \beta (x), \ldots, \alpha(s) \, \partial_d \beta (x)) \in \mathbb{R}^d$.
Then, the \textit{mollification}, 
$A_{\alpha, \beta, t}$, of $\nabla q - \nabla p$ given by 
\begin{equation*}
  A_{\alpha,\beta,t}(x) := ((\alpha \otimes \beta) 
  * \nabla q - (\alpha \otimes \beta) * \nabla p) (t,x)  
\end{equation*}
is well-defined on $(\varepsilon_0, T-\varepsilon_0) \times \mathbb{R}^d$.   
Furthermore, for all $t \in (\varepsilon_0, T-\varepsilon_0)$, 
\begin{align}\label{Eq:convol} 
  A_{\alpha,\beta,t} 
  & =((-\partial_t \, \alpha \otimes \beta + \alpha \otimes \Delta\beta) 
    * {\textbf{u}})(t,\cdot ) 
    + ((\alpha \otimes \nabla \beta)  
    * (- {\textbf{u}} \otimes {\textbf{u}} +\mathbb{F}))(t,\cdot )  
    \nonumber \\
  & \;\;\;\; -((\alpha \otimes \nabla \beta) * p) (t,\cdot),
\end{align} 
where for every $s \in (\varepsilon_0, T-\varepsilon_0)$, and for almost 
all $x \in \mathbb{R}^d$, $((\alpha \otimes \nabla \beta) * 
(-{\textbf{u}} \otimes {\textbf{u}} +\mathbb{F}))(s,x)$ denotes the 
vector whose i-th coordinate is given by $\displaystyle \sum_{j=1}^d 
((-u_i u_j + F_{i,j}) * (\alpha \otimes \partial_i \beta)(s,x)$, for 
each $i \in \{1, \ldots, d\}$. For almost all 
$t \in (\varepsilon_0, T-\varepsilon_0)$, we will prove 
that $A_{\alpha,\beta,t}$ is necessarily zero, by demonstrating that it is a 
harmonic tempered distribution, and therefore it is a polynomial, and finally 
by showing that $A_{\alpha,\beta,t}$ belongs to a space which does not contain 
nontrivial polynomials.  Since $\gamma \in (0, d)$, and 
\begin{equation*}
  \lim_{r \to \, 1^+} \frac{d(2-r)}{2r} + \frac{\gamma}{2} 
  = \frac{d}{2} + \frac{\gamma}{2} < d, 
\end{equation*}
there exist real numbers $\sigma$ and $\eta$ such that
\begin{equation}  \label{Eq:5}
     1 < \sigma < \min \Big{\{} \frac{d}{d-1}, \frac{d}{\gamma} \Big{\}}, 
\end{equation}
and
\begin{equation}  \label{Eq:6}
  \max \Big{\{} \gamma , \, \frac{d(2-\sigma)}{2\sigma}
  +\frac{\gamma}{2} \Big{\}} < \, \eta < \, \frac{d}{\sigma} \, < \, d.
\end{equation}
For almost every $t \in (0,T)$, by using (\ref{Eq:1}), (\ref{Eq:3}) 
and (\ref{Eq:5}), we see that  
\begin{equation*}
  p(t) \, \in \, L^{\sigma}_{w_{\sigma\gamma}}(\mathbb{R}^d) 
  + L^2_{w_\gamma}(\mathbb{R}^d), 
\end{equation*} 
and, of course, by employing (\ref{Eq:0b}) and (\ref{Eq:5}), we get
\begin{equation*}
  (-{\textbf{u}} \otimes {\textbf{u}} +\mathbb{F})(t) 
  \, \in \, L^{\sigma}_{w_{\sigma \gamma}}(\mathbb{R}^d) 
  + L^2_{w_\gamma}(\mathbb{R}^d).
\end{equation*}  
Since for all function $\varphi \in \mathcal{C}_0^{\infty}(\mathbb{R}^d)$, and for 
every locally integrable function $f:\mathbb{R}^d \to \mathbb{C}$, 
$\vert f*\varphi\vert\le C_\varphi \, \mathcal{M} f$, it follows that 
the convolution with a function belonging to 
$\mathcal{C}_0^{\infty}(\mathbb{R}^d)$ is a bounded linear operator defined on 
$L^2_{w_\gamma}(\mathbb{R}^d)$, and also on 
$L^{\sigma}_{w_{\sigma \gamma}}(\mathbb{R}^d)$. 
Thus, from (\ref{Eq:convol}), for almost every 
$t \in (\varepsilon_0, T-\varepsilon_0)$, we obtain that 
\begin{equation*}
  A_{\alpha,\beta,t} \in L^{\sigma}_{w_{\sigma \gamma}}(\mathbb{R}^d) 
  + L^2_{w_\gamma}(\mathbb{R}^d), 
\end{equation*}  
and therefore $A_{\alpha,\beta,t}$ is a tempered  distribution.  Moreover, for 
all $t \in (\varepsilon_0, T-\varepsilon_0)$, $A_{\alpha,\beta,t}$ is a harmonic 
polynomial. In fact, let $s \in (\varepsilon_0, T-\varepsilon_0)$. Then, 
\begin{equation*}
  \Delta A_{\alpha, \beta, s}=(\alpha \otimes \beta)*
  \nabla (\Delta (q-p))(s, \cdot)=0,
\end{equation*}  
where in the last expression we have used (\ref{Eq:Lapl}).  In addition, by 
using H\"older's inequality, we obtain that 
\begin{equation*}
  L^{\sigma}_{w_{\sigma \gamma}}(\mathbb{R}^d) +L^2_{w_\gamma}(\mathbb{R}^d) 
  \subset L^{\sigma}_{w_{\sigma \eta}}(\mathbb{R}^d). 
\end{equation*}
In fact, for all 
$g \in L^{\sigma}_{w_{\sigma \gamma}}(\mathbb{R}^d) + L^2_{w_\gamma}(\mathbb{R}^d)$, 
we see that 
\begin{align*}
  \int |g(x)|^{\sigma} w_{\sigma \eta}(x) \, dx 
  & \le \left(\int (|g(x)|^{\sigma} 
    w_{\frac{\sigma \gamma}{2}}(x))^{\frac{2}{\sigma}} \, dx 
    \right)^{\frac{\sigma}{2}} \\ 
  &\;\;\;\;  \cdot \left( \int (1+|x|)^
    {-(\sigma \eta \, - \, \frac{\sigma \gamma}{2})(\frac{2}{2-\sigma})} 
    \, dx \right)^{1-\frac{\sigma}{2}}  \\
  & < + \infty,    
\end{align*}
where in the last expression we have used the fact that 
$\displaystyle \left(\eta - \frac{\gamma}{2}\right) \cdot 
\left(\frac{2\sigma}{2-\sigma}\right) > d$, which follows from (\ref{Eq:6}). 
By using (\ref{Eq:6}), we see that $\sigma \eta < d$. Then, the space 
$L^{\sigma}_{w_{\sigma \eta}}(\mathbb{R}^d)$ does not contain nontrivial 
polynomials.  Hence, for almost all $t \in (\varepsilon_0, T-\varepsilon_0)$, 
$A_{\alpha,\beta,t}= 0$. For all $\varepsilon > 0$, for every 
$t \in \mathbb{R}$, and for each $x \in \mathbb{R}^d$, we now let 
\begin{equation*}
  \alpha_\varepsilon (t) := \frac {1}{\varepsilon} 
  \alpha \left( \frac{t}{\varepsilon} \right),
  \;\;\; \text{ and } \;\;\;  
  \beta_\varepsilon (x) := \frac {1}{\varepsilon^d} 
  \beta \left( \frac {x}\varepsilon \right). 
\end{equation*}
Since for any $\varepsilon > 0$, $\alpha_\varepsilon \otimes \beta_\varepsilon 
\in \mathcal{C}_0^{\infty}(\mathbb{R} \times \mathbb{R}^d)$ is a 
nonnegative real-valued function satisfying
\begin{equation*}
  \int_{\mathbb{R} \times \mathbb{R}^d} 
  (\alpha_\varepsilon \otimes \beta_\varepsilon)(s, x) \, ds \, dx = 1, 
\end{equation*}  
we conclude that for almost all $t \in (0, T)$, 
\begin{equation*}
  \nabla (q-p)(t,\cdot ) = \lim_{\varepsilon \to \, 0} 
  A_{\alpha_\varepsilon, \beta_\varepsilon, t} = 0.
\end{equation*}
\end{proof}


\end{document}